\documentclass[11pt, a4paper, reqno]{amsart}
\usepackage[utf8]{inputenc}
\usepackage{amsmath,amssymb}
\usepackage{graphicx}
\usepackage{indentfirst,csquotes}
\usepackage[colorlinks=true, linkcolor=blue, citecolor=blue, urlcolor=blue]{hyperref}
\usepackage{pdflscape} 
\usepackage{hhline} 
\usepackage{array} 
\usepackage{calrsfs}
\DeclareMathAlphabet{\pazocal}{OMS}{zplm}{m}{n}
\usepackage{latexsym,amsmath,amsfonts,amscd,amssymb, mathrsfs, amsthm}
\usepackage[english]{babel}
\usepackage{appendix}
\usepackage{hyperref}
\usepackage{mathtools}
\usepackage{booktabs}
\usepackage{natbib}

\advance\oddsidemargin by -0.5cm
\advance\evensidemargin by -0.5cm
\advance\topmargin by -1.5cm
\usepackage{booktabs}
\usepackage[scale=2]{ccicons}

\usepackage{epic}
\usepackage{autograph}
\usepackage{tikz}
\usetikzlibrary{arrows,decorations.markings}

\usepackage[normalem]{ulem}

\theoremstyle{definition}
\def\mxl{\left[ \begin{array}}  \def\mxr{\end{array} \right]}
\def\detl{\left| \begin{array}}  \def\detr{\end{array} \right|}
\def\bmx{\begin{bmatrix}} \def\emx{\end{bmatrix}}

\def\bc{\begin{center}}          \def\ec{\end{center}}
\def\ben{\begin{enumerate}}      \def\een{\end{enumerate}}
\def\beq{\begin{equation}}       \def\eequ{\end{equation}}
\def\beqn{\begin{eqnarray*}} \def\eeqn{\end{eqnarray*}} 
\def\eqs{\vspace{-10pt}\begin{equation}} \def\eqe{\vspace{-5pt}\end{equation}}
\def\bal{\begin{align}}           \def\eal{\end{align}}
\def\ben{\begin{enumerate}}       \def\een{\end{enumerate}}
\def\bit{\begin{itemize}}           \def\eit{\end{itemize}}
\def\btabb{\begin{tabbing}}         \def\etabb{\end{tabbing}}
\def\btab{\begin{tabular}}          \def\etab{\end{tabular}}
\def\barr{\begin{array}}           \def\earr{\end{array}}
\def\earrb{\end{array} \right\}}

\def\bf{\textbf}  

\def\0{\bf{0}}    

\font\Bbb=msbm10 
\def\A{\ifmmode{\mbox{\Bbb A}}\else{{\Bbb A}}\fi}
\def\B{\ifmmode{\mbox{\Bbb B}}\else{{\Bbb B}}\fi}
\def\R{\ifmmode{\mbox{\Bbb R}}\else{{\Bbb R}}\fi}
\def\C{\ifmmode{\mbox{\Bbb C}}\else{{\Bbb C}}\fi}

\def\G{\ifmmode{\mbox{\Bbb G}}\else{{\Bbb G}}\fi}
\def\N{\ifmmode{\mbox{\Bbb N}}\else{{\Bbb N}}\fi}
\def\Z{\ifmmode{\mbox{\Bbb Z}}\else{{\Bbb Z}}\fi}
\def\K{\ifmmode{\mbox{\Bbb K}}\else{{\Bbb K}}\fi}
\def\L{\ifmmode{\mbox{\Bbb L}}\else{{\Bbb L}}\fi}
\def\D{\ifmmode{\mbox{\Bbb D}}\else{{\Bbb D}}\fi}
\def\T{\ifmmode{\mbox{\Bbb T}}\else{{\Bbb T}}\fi}
\def\I{\ifmmode{\mbox{\Bbb I}}\else{{\Bbb I}}\fi}
\def\E{\ifmmode{\mbox{\Bbb E}}\else{{\Bbb E}}\fi}
\def\K{\ifmmode{\mbox{\Bbb K}}\else{{\Bbb K}}\fi}
\def\Q{\ifmmode{\mbox{\Bbb Q}}\else{{\Bbb Q}}\fi}
\def\P{\ifmmode{\mbox{\Bbb P}}\else{{\Bbb P}}\fi}

\def\cG{\ifmmode{\mathcal{G}}\else{{$\mathcal{G}$}}\fi} 
\def\cA{\ifmmode{\mathcal{A}}\else{{$\mathcal{A}$}}\fi}

\def\op{ \ifmmode{\oplus} \else{\leavevmode\hbox{$\oplus$}\fi } }
\def\ot{\ifmmode{\otimes}\else{$\!\!\otimes$}\fi}
\def\od{ \ifmmode{\odot} \else{\leavevmode\hbox{$\odot$}\fi } }
\def\adots{\mathinner{\raise 1pt\hbox{.}\mkern1mu\raise4 pt\hbox{.}\mkern2mu
  \mkern1mu\raise7 pt\vbox{\kern7 pt\hbox{.}}\mkern2mu}}

\def\bdefn{\begin{defn}} \def\edefn{\end{defn}}
\def\bex{\begin{ex}} \def\eex{\end{ex}}
\def\beg{\begin{eg}} \def\eeg{\end{eg}}
\def\blem{\begin{lem}} \def\elem{\end{lem}}
\def\bth{\begin{thm}} \def\eth{\end{thm}}
\def\bprop{\begin{prop}} \def\eprop{\end{prop}}
\def\bcor{\begin{cor}} \def\ecor{\end{cor}}

\newtheorem{thm}{Theorem}[section]  \newtheorem{lem}{Lemma}[section]
\newtheorem{cor}{Corollary}[section]  \newtheorem{prop}{Proposition}[section]
\newtheorem{eg}{Example}[section]    \newtheorem{ex}{Exercise}[section]
\newtheorem{defn}{Definition}[section]

\def\ball{\begin{align}}  \def\eall{\end{align}}

\def\bbin{\left( \begin{array}{c}} \def\ebin{\end{array} \right)}

\usepackage{pgfplots}
\pgfplotsset{compat=1.18}

\begin{document}
\title[Drazin Inverses and Walk Structure for Dutch Windmill Digraphs]{Drazin inverses and walk structure of oriented Dutch windmill graphs}
\author[C. Mendes Araújo]{C. Mendes Araújo}
\author[Faustino Maciala]{Faustino Maciala}
\author[Pedro Patrício]{Pedro Patrício}
\date{\today}
\address{C. Mendes Araújo, CMAT -- Centro de Matemática and Departamento de Matemática Universidade do Minho 4710-057 Braga Portugal}
\email{clmendes@math.uminho.pt}
\address{Faustino Maciala, CMAT -- Centro de Matemática Universidade do Minho 4710-057 Braga Portugal; Departamento de Ciências da Natureza e Ciências Exatas do Instituto Superior de Ciências da Educação de Cabinda -- ISCED-Cabinda, Angola.}
\email{fausmacialamath@hotmail.com}
\address{Pedro Patrício, CMAT -- Centro de Matemática and Departamento de Matemática Universidade do Minho 4710-057 Braga Portugal}
\email{pedro@math.uminho.pt}

\begin{abstract}
We investigate the Drazin invertibility of adjacency matrices associated with a class of oriented graphs known as oriented Dutch windmill graphs. By analyzing walks of prescribed lengths and exploiting the structure of the minimal polynomial, we obtain explicit expressions for the Drazin inverse and determine its index. The approach combines combinatorial enumeration with algebraic matrix analysis, offering a constructive characterization that generalizes known results for paths, cycles, and bipartite graphs. Beyond its intrinsic theoretical value, the framework provides insight into discrete models governed by cyclic feedback and may serve as a basis for symbolic computation of generalized inverses in structured networks.
\end{abstract} 

\date{\today}
\subjclass{15A09, 15A10, 05C20}
\keywords{Drazin inverse, Drazin index, minimal polynomial, directed graph}
\maketitle
\medskip

 \section{Introduction}

\hspace{0.5cm}A fundamental problem in discrete mathematics and spectral graph theory is to analyze the structure and enumeration of walks in digraphs, particularly those with cyclic or regular patterns. Directed graphs with repeated cycles arise in network dynamics, symbolic systems, flow models, and discrete-time feedback processes. Their adjacency matrices translate combinatorial connectivity into algebraic form, providing a natural bridge between graph theory and matrix analysis.

Within this framework, determining the (pseudo-)invertibility of matrices associated with graphs, and obtaining explicit expressions for such inverses, plays a crucial role. Classical results include the study of the inverse of the distance matrix of wheel graphs (\cite{Balaji}), 
 as well as the invertibility of adjacency matrices of trees (\cite{Godsil})
and bipartite graphs (\cite{Bapat,McLeman}).
 These investigations have been extended to generalized inverses, including the group and Drazin inverses, for various classes of graphs (\cite{Catral2,Catral-grp,McDonald,Sivakumar}).
  A general matrix-based approach to graph theory is presented in \cite{bookBapat}.

The study of generalized inverses of adjacency matrices, notably the Drazin inverse, offers insight into singular systems, absorbing Markov chains, and differential equations on networks. Although explicit Drazin inverses are known for paths, cycles, and bipartite graphs, deriving closed-form expressions for more complex digraphs remains a challenge.

Oriented Dutch windmill graphs, constructed by coalescing multiple directed cycles at a common vertex, combine recursive structure with combinatorial richness. Their simple recursive structure enables exact enumeration of walks, while still presenting nontrivial behavior from an algebraic perspective. Despite being a restricted class, they serve as a useful model to explore how local cyclic patterns influence pseudo-invertibility, and how generalized inverses capture walk-based properties. In this work, we compute the Drazin inverse of their adjacency matrices, establish their Drazin index, and characterize the structure of walks of fixed length. Our method is constructive and combinatorial: we characterize the minimal polynomial, compute the Drazin index, and identify the precise vertex pairs connected by unique walks of a fixed length. These results reveal how the combinatorial structure of the graph directly determines the support of the Drazin inverse, offering insight into the interplay between graph topology and matrix pseudo-invertibility.

Throughout this paper, we consider real matrices. The Drazin inverse of an $n \times n$ matrix $A$, denoted by $A^D$, is the unique solution to the equations $A^{k+1}X = A^k$, $XAX=A$, and $AX=XA$, for some non-negative integer $k$. The least such $k$ is called the Drazin index of $A$, written $\operatorname{ind}(A)$. When $\operatorname{ind}(A) \leq 1$, $A$ is said to have a group inverse, denoted $A^\#$. General theory on Drazin and group inverses can be found in \cite{benisraelgreville, campbell1979}.

Let $\psi_A$ and $\Delta_A$ denote, respectively, the minimal and characteristic polynomials of $A$. It is well known that:
$$
\psi_A \mid \Delta_A \mid \psi_A^n,
$$
which implies that $\psi_A$ and $\Delta_A$ share the same irreducible factors. If $0$ is an eigenvalue of $A$, its multiplicity as a root of $\psi_A$ is called the index of $A$, denoted $i(A)$, and it coincides with the Drazin index: $\operatorname{ind}(A) = i(A)$. When $\psi_A(\lambda) = \lambda^{i(A)} f(\lambda)$, with $\gcd(\lambda, f(\lambda)) = 1$, the Drazin inverse of $A$ may be expressed in terms of powers of $A$ and a polynomial in $A$, yielding practical computational methods.

 Let $G$ be a directed graph with vertex set $\{v_1, v_2, \ldots, v_n\}$. Its adjacency matrix $A = (a_{ij}) \in \mathbb{R}^{n \times n}$ is defined by
$$
a_{ij} = \begin{cases}
1 & \text{if there is an edge from } v_i \text{ to } v_j \text{ in } G, \\
0 & \text{otherwise}.
\end{cases}
$$

A walk in $G$ is a sequence $P = \langle v_1, v_2, \ldots, v_r \rangle$ of vertices such that each consecutive pair is connected by a directed edge. The length of $P$, denoted by $|P|$, is equal to $r-1$, which corresponds to the number of edges in the walk. If $Q = \langle w_1, \ldots, w_s \rangle$ is another walk with $w_1 = v_r$, we denote by $Q\circ P$ the walk in $G$ obtained by concatenating $P$ and $Q$, that is
$$Q\circ P=\langle v_1,v_2,\ldots,v_r,w_2,\ldots,w_s\rangle.$$ Observe that $|Q\circ P|=|P|+|Q|$.

In the sections that follow, we identify an annihilating polynomial of the adjacency matrices of oriented Dutch windmill graphs, determine their Drazin index, and deduce a formula for the Drazin inverse. Our method is constructive and combinatorial, rooted in the analysis of walks and their relation to matrix powers.

\section{Oriented Dutch windmill graphs}  

\hspace{0.5cm} The oriented Dutch windmill graph $D^m_n$ is the directed graph constructed by taking $m$ copies of the directed cycle $C_n$ (with $n \geq 3$) and identifying a single common vertex among all copies. Figure \ref{fig:dutch-windmill-d34} shows the digraph $D_3^4$, obtained by taking four copies of the directed cycle of length~$3$ sharing a common vertex.

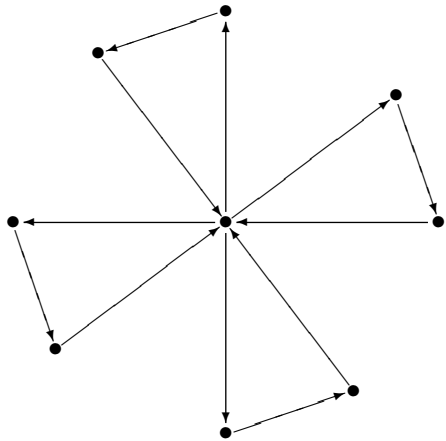
\begin{figure}[h]
\centering
\setlength{\unitlength}{0.28cm}
\begin{picture}(24,24)(-12,-12)

\letvertex[1] A1=(0,0)      \drawvertex(A1){$\bullet$}
\letvertex[1] A2=(8,6)      \drawvertex(A2){$\bullet$}
\letvertex[1] A3=(10,0)     \drawvertex(A3){$\bullet$}
\letvertex[1] A4=(0,10)     \drawvertex(A4){$\bullet$}
\letvertex[1] A5=(-6,8)     \drawvertex(A5){$\bullet$}
\letvertex[1] A6=(-10,0)    \drawvertex(A6){$\bullet$}
\letvertex[1] A7=(-8,-6)    \drawvertex(A7){$\bullet$}
\letvertex[1] A8=(0,-10)    \drawvertex(A8){$\bullet$}
\letvertex[1] A9=(6,-8)     \drawvertex(A9){$\bullet$}

\drawedge(A1,A2){}
\drawedge(A2,A3){}
\drawedge(A3,A1){}

\drawedge(A1,A4){}
\drawedge(A4,A5){}
\drawedge(A5,A1){}

\drawedge(A1,A6){}
\drawedge(A6,A7){}
\drawedge(A7,A1){}

\drawedge(A1,A8){}
\drawedge(A8,A9){}
\drawedge(A9,A1){}

\end{picture}
\caption{Oriented Dutch windmill graph $D_3^4$.}
\label{fig:dutch-windmill-d34}
\end{figure}

It is well known that the Drazin index is invariant under similarity transformations, since similar matrices share the same minimal polynomial. Moreover, the Drazin inverse itself transforms compatibly with similarity: if $B = UAU^{-1}$, then $B^D = UA^D U^{-1}$. 

In this context, when we assume that the graph $D_n^m$ has a given set of vertices and edges, we are implicitly considering a specific labeling of the vertices. However, different labelings of the same graph lead to permutation-similar adjacency matrices. Since similarity - including permutation similarity - preserves the minimal polynomial and the Drazin inverse, all adjacency matrices corresponding to isomorphic copies of $D_n^m$ yield the same spectral and Drazin properties. Therefore, our analysis remains valid regardless of the particular labeling chosen.

Without loss of generality, we shall henceforth consider the oriented Dutch windmill graph $D^m_n = (V, E)$ defined by
$$
V = \{1, 2, \ldots, m(n - 1) + 1\},
$$
and edge set $E$ consisting of directed edges $(a, b)$ determined as follows:
\begin{equation*}
\begin{split}
& a = 1\ \textrm{ and }\ b=(k-1)(n-1)+2, \textrm{ for }\ k\in\{1,\ldots, m\}\\
\textrm{or} & \\
 & a =(k-1)(n-1)+i\ \textrm{ and }\ b=a+1, \textrm{ for }\ k\in\{1,\ldots, m\} \textrm{ and }\ i\in\{2,\ldots, n-1\}\\
 \textrm{or} & \\
& a=(k-1)(n-1)+n\ \textrm{ and }\ b=1, \textrm{ for }\ k\in\{1,\ldots, m\}.
\end{split}
\end{equation*}

We begin by  introducing some notation and preliminary results concerning walks of specific lengths in the graph $D^m_n$.

For each $k \in \{1, \ldots, m\}$, let $C^k$ denote the directed cycle
$$\Big\langle 1,(k-1)(n-1)+2,\ldots,(k-1)(n-1)+(n-1), (k-1)(n-1)+n,1\Big\rangle,$$
which is a cycle of length $n$, beginning and ending at vertex 1. Observe that the cycles $C^1, C^2, \ldots, C^m$ represent the $m$ copies of the directed cycle $C_n$, all sharing the common vertex 1. Let $V^k$ denote the set of vertices comprising the cycle $C^k$, for each $k = 1, 2, \ldots, m$.

\begin{lem}\label{walksn-1}
 Let $k\in\{1,\ldots,m\}$. A walk of length $n-1$ from vertex $i$ to vertex $j$, with $i,j\in V^k$, exists in $D^m_n$ if and only if one of the following conditions holds:
$$i=1\ \textrm{ and }\ j=(k-1)(n-1)+n,$$
or
$$i=(k-1)(n-1)+2\ \textrm{ and }\ j=1,$$
or 
$$i=(k-1)(n-1)+\ell\ \textrm{ and }\ j=i-1 \textrm{ for some }\ \ell\in\{3,\ldots, n\}.$$
Moreover, in each of these cases, there exists exactly one such walk of length $n-1$ from $i$ to $j$ in $D^m_n$.
\end{lem}

\proof Let  $i,j\in V^k$. 

If $i=j$, the shortest walk from $i$ to $j$ in $D^m_n$ is the trivial walk $P=\langle i\rangle$, which has length $|P|=0$.  If $i=j=1$, any nontrivial walk from $i$ to $j$ in $D^m_n$ can be expressed as $$C^{r_s}\circ\ldots\circ C^{r_1},$$ where $s\in\mathbb N$ and $r_1,\ldots, r_s\in\{1,\ldots,m\}$. Each such walk has length $sn\geq n$. If $i=j\neq 1$, any nontrivial walk from $i$ to $j$ in $D^m_n$ must pass through vertex $1$ and can be described either as $$Q_{1i}\circ Q_{i1}$$ or $$Q_{1i}\circ C^{r_s}\circ\ldots\circ C^{r_1}\circ Q_{i1},$$ where $Q_{i1}$ and $Q_{1i}$ are the shortest paths within $C^k$ from $i$ to $1$ and from $1$ to $i$, respectively, $s\in\mathbb N$ and $r_1,\ldots, r_s\in\{1,\ldots,m\}$. These walks have lengths $n$ and $(s+1)n>n$, respectively.

Now, assume $i\neq j$. We distinguish three cases:

\noindent (I) $i=1$

Since $j\in V^k$, there exists $\ell\in\{2,...,n\}$ such that $j=(k-1)(n-1)+\ell$. The shortest walk from $1$ to $j$ in $D^m_n$ is 
$$P_{1j}=\Big\langle i=1,(k-1)(n-1)+2,\ldots,(k-1)(n-1)+(\ell-1), (k-1)(n-1)+\ell=j\Big\rangle,$$
of length $\ell -1$. Any other walk from $1$ to $j$ in $D^m_n$ can be written as $$P_{1j}\circ C^{r_s}\circ\ldots\circ C^{r_1},$$ where $s\in\mathbb N$ and $r_1,\ldots, r_s\in\{1,\ldots,m\}$, with total length $sn+|P|>n$. Hence, there exists a walk in $D^m_n$ of length $n-1$ from $i=1$ to $j$ if and only if $\ell=n$, i.e., $j=(k-1)(n-1)+n$. In that case, $P_{1j}$ is the unique walk in $D^m_n$ of length $n-1$ from $i=1$ to $j$.

\noindent (II) $j=1$

Since $i\in V^k$, there exists $\ell\in\{2,...,n\}$ such that $i=(k-1)(n-1)+\ell$. The shortest walk from $i$ to $1$ in $D^m_n$ is 
$$P_{i1}=\Big\langle i=(k-1)(n-1)+\ell,(k-1)(n-1)+(\ell+1), \ldots,(k-1)(n-1)+n,1=j\Big\rangle,$$
with length $n-\ell+1$. Any other walk from $i$ to $1$ in $D^m_n$ can be expressed as $$C^{r_s}\circ\ldots\circ C^{r_1}\circ P_{i1},$$ where $s\in\mathbb N$ and $r_1,\ldots, r_s\in\{1,\ldots,m\}$, and has length $sn+(n-\ell+1)>n$. Therefore, there exists a walk of length $n-1$ from $i$ to $j=1$ if and only if $\ell=2$, i.e., $i=(k-1)(n-1)+2$. In this case, it is unique.

\noindent (III) $i\neq 1$ and $j\neq 1$.

Let $$i=(k-1)(n-1)+\ell,\quad j=(k-1)(n-1)+\ell^\prime,$$ for some $\ell,\ell^\prime\in\{2,\ldots, n\}$, with $\ell\neq\ell^\prime$. 

If $\ell<\ell^\prime$, the shortest walk from $i$ to $j$ in $D^m_n$ is 
\begin{align*}P_{ij}=&\Big\langle i=(k-1)(n-1)+\ell,(k-1)(n-1)+(\ell+1), \ldots,\\&(k-1)(n-1)+(\ell^\prime-1),(k-1)(n-1)+\ell^\prime=j\Big\rangle,\end{align*}
with length $\ell^\prime-\ell<n-1$. Any other walk from $i$ to $j$ in $D^m_n$ can be described as 
$$P_{ij}\circ Q_{1i}\circ Q_{i1}$$ or as $$P_{ij}\circ Q_{1i}\circ C^{r_s}\circ\ldots\circ C^{r_1}\circ Q_{i1},$$ where $Q_{i1}$ and $Q_{1i}$ are the shortest paths in $C^k$ from $i$ to $1$ and from $1$ to $i$, respectively, $s\in\mathbb N$ and $r_1,\ldots, r_s\in\{1,\ldots,m\}$. These walks have lengths $n+|P_{ij}|>n$ and $(s+1)n+|P_{ij}|>n$, respectively. Therefore, if $\ell<\ell^\prime$, there is no walk of length $n-1$ from $i$ to $j$.

If $\ell>\ell^\prime$, the shortest walk from $i$ to $j$ in $D^m_n$ passes through vertex $1$ and is given by
\begin{align*}
P_{ij}=&\Big\langle i=(k-1)(n-1)+\ell,(k-1)(n-1)+(\ell+1),\ldots, (k-1)(n-1)+n,1,\\
&(k-1)(n-1)+2,\ldots, (k-1)(n-1)+\ell^\prime=j \Big\rangle,\end{align*}
with total length $(n-\ell+1)+(\ell^\prime-1)=n-\ell+\ell^\prime$. Any other walk from $i$ to $j$ in $D^m_n$ can be described as 
$$Q_{1j}\circ C^{r_s}\circ\ldots\circ C^{r_1}\circ Q_{i1},$$ where $Q_{i1}$ and $Q_{1j}$ are the shortest paths in $C^k$ from $i$ to $1$ and from $1$ to $j$, respectively, $s\in\mathbb N$ and $r_1,\ldots, r_s\in\{1,\ldots,m\}$. The length of each of such walks from $i$ to $j$ is $sn+|P_{ij}|>n$. Therefore, if $\ell>\ell^\prime$, there exists a walk from $i$ to $j$ in $D^m_n$ of length $n-1$ if and only if $\ell^\prime=\ell-1$. In this case, $P_{ij}$ is the unique walk in $D^m_n$ of length $n-1$ from $i$ to $j$. \endproof

\textit{The existence and uniqueness of such walks highlight the strict orientation and labeling of each cycle. The previous lemma lays the combinatorial foundation for identifying the exact support of the Drazin inverse.}

\begin{lem}  \label{walksn-1a} Let $k,r\in\{1,\ldots,m\}$ with $k\neq r$. There exists a walk of length $n-1$ in $D^m_n$ from a vertex $i\in V^k\setminus\{1\}$ to a vertex $j\in V^r\setminus\{1\}$ if and only if there exists $\ell\in\{3,\ldots, n\}$ such that 
$$i=(k-1)(n-1)+\ell\ \textrm{ and }\ j=(r-1)(n-1)+(\ell-1).$$ 
Moreover, in such cases, this walk is unique.
\end{lem}

\proof
Since $i\in V^k\setminus\{1\}$ and $j\in V^r\setminus\{1\}$, there exist $\ell,\ell^\prime\in\{2,\ldots, n\}$
such that 
$$i=(k-1)(n-1)+\ell\ \textrm{ and }\ j=(r-1)(n-1)+\ell^\prime.$$ 
Consider the walk
\begin{align*}P_{ij}=&\Big\langle i=(k-1)(n-1)+\ell,(k-1)(n-1)+(\ell+1),\ldots, (k-1)(n-1)+n,1,\\
&(r-1)(n-1)+2,\ldots, (r-1)(n-1)+\ell^\prime =j\Big\rangle.\end{align*}
This walk traverses from $i$ to $1$ within $C^k$, then from $1$ to $j$ within $C^r$. It is the shortest walk from $i$ to $j$ in $D^m_n$, with length $(n-\ell+1)+(\ell^\prime-1)=n-\ell+\ell^\prime$. For this to equal $n-1$, we must have $\ell\geq 3$ and $\ell^\prime=\ell-1$. Any other walk from $i$ to $j$ in $D^m_n$ can be expressed as
$$Q_{1j}\circ C^{r_s}\circ\ldots\circ C^{r_1}\circ Q_{i1},$$ where $Q_{i1}$ is the shortest path in $C^k$ from $i$ to $1$, $Q_{1j}$ is the shortest path in $C^r$ from $1$ to $j$, $s\in\mathbb N$ and $r_1,\ldots, r_s\in\{1,\ldots,m\}$. The length of each of such walks from $i$ to $j$ is $sn+|P_{ij}|>n$. Hence, there is a walk in $D^m_n$ of length $n-1$ from $i$ to $j$   if and only if $\ell\geq 3$ and $\ell^\prime=\ell-1$, and in that case, that walk, $P_{ij}$, is unique. \endproof

\begin{lem} For any pair of vertices $i, j \in V$, if $P$ is a walk of length $2n - 1$ in $D^m_n$ from $i$ to $j$, then exactly one of the following cases occurs:
\begin{enumerate}
\item[(a)] $i=1$, and there exists $k\in\{1,\ldots, m\}$ such that $j=(k-1)(n-1)+n$. In this case, the walk $P$ can be written as
$$P=P_{1j}\circ C^s,$$
where $s\in\{1,\ldots, m\}$ and $P_{1j}$ is the unique walk in $D^m_n$ of length $n-1$ from $i$ to $j$;
\item[(b)] $j=1$, and there exists $k\in\{1,\ldots, m\}$ such that $i=(k-1)(n-1)+2$. In this case, the walk $P$ can be expressed as
$$P=C^s\circ P_{i1},$$
where $s\in\{1,\ldots, m\}$ and $P_{i1}$ is the unique walk in $D^m_n$ of length $n-1$ from $i$ to $j$;
\item[(c)] There exist $k,r\in\{1,\ldots, m\}$ and $\ell\in\{3,\ldots, n\}$ such that $i=(k-1)(n-1)+\ell$, $j=(r-1)(n-1)+(\ell-1)$ and the walk $P$ can be written as
$$P=Q_{1j}\circ C^s\circ Q_{i1},$$
where $s\in\{1,\ldots, m\}$, and the concatenation $Q_{1j}\circ Q_{i1}$ forms the unique walk of length $n-1$ from $i$ to $j$  in $D^m_n$.
\end{enumerate}
\end{lem}

\proof Any walk $P$ in $D^m_n$ of length $2n - 1$ must include exactly one of the cycles $C^1,\ldots,C^m$, and the remaining $n-1$ edges must form a walk from $i$ to $j$. Thus, $P$ decomposes as a cycle $C^s$ of length $n$, together with a walk of length $n - 1$ connecting $i$ to $j$.

All possible walks of length $n-1$ in $D^m_n$ are classified in Lemmas \ref{walksn-1} and \ref{walksn-1a}. Each of the three cases (a)-(c) corresponds to one of the possible types of such walks, and the inclusion of the cycle accounts for the additional $n$ edges, bringing the total length to $2n-1$.

Therefore, each walk of length $2n-1$ from $i$ to $j$ must fall into exactly one of the described forms. \endproof

The following result is an immediate consequence of the preceding lemmas and of the fact that the digraph $D_n^m$ contains exactly $m$ copies of the directed cycle $C_n$.

\begin{lem} \label{walks2n-1}
Let $i,j\in V$. There exists a walk of length $2n-1$ from vertex $i$ to vertex $j$ in $D^m_n$ if and only if there exists a walk of length $n-1$ from $i$ to $j$ in $D^m_n$.
In such cases, there are exactly $m$ distinct walks of length $2n-1$ from $i$ to $j$ in $D^m_n$.
\end{lem}

These last two lemmas admit a straightforward generalization, as stated below.

\begin{lem} Let $i, j \in V$ and $p \in \mathbb{N}$ with $p \geq 2$. If $P$ is a walk in $D^m_n$ of length $pn - 1$ from vertex $i$ to vertex $j$, then exactly one of the following holds:
\begin{enumerate}
\item[(a)] $i=1$, and there exists $k\in\{1,\ldots, m\}$ such that $j=(k-1)(n-1)+n$. In this case, $P$ can be expressed as
$$P=P_{1j}\circ C^{r_{p-1}}\circ\ldots\circ C^{r_1},$$
where $r_1,\ldots, r_{p-1}\in\{1,\ldots, m\}$, and $P_{1j}$ is the unique walk of length $n-1$ from $i$ to $j$ in $D^m_n$;
\item[(b)] $j=1$, and there exists $k\in\{1,\ldots, m\}$ such that $i=(k-1)(n-1)+2$. Then, $P$ can be written as
$$P=C^{r_{p-1}}\circ\ldots\circ C^{r_1}\circ P_{i1},$$
where $r_1,\ldots, r_{p-1}\in\{1,\ldots, m\}$ and $P_{i1}$ is the unique walk of length $n-1$ from $i$ to $j$ in $D^m_n$;
\item[(c)] There exist $k,r\in\{1,\ldots, m\}$ and $\ell\in\{3,\ldots, n\}$ such that $i=(k-1)(n-1)+\ell$, $j=(r-1)(n-1)+(\ell-1)$ and $P$ can be expressed as
$$P=Q_{1j}\circ C^{r_{p-1}}\circ\ldots\circ C^{r_1}\circ Q_{i1},$$
where $r_1,\ldots, r_{p-1}\in\{1,\ldots, m\}$, and $Q_{1j}\circ Q_{i1}=P_{ij}$, the unique walk of length $n-1$ from $i$ to $j$ in $D^m_n$ .
\end{enumerate}
\end{lem}

\proof To construct a walk $P$ of length $pn-1$ from $i$ to $j$, one must concatenate $p-1$ cycles of length $n$, followed by a walk of length $n - 1$ from $i$ to $j$. The existence and structure of such walks of length $n-1$ are given by Lemmas \ref{walksn-1} and \ref{walksn-1a}, and all combinations involving $p-1$ cycles yield valid walks. The three cases correspond to all possible configurations that admit such a walk.\endproof

\begin{lem} \label{walkspn-1}
Let $i,j\in V$ and let $p\in\mathbb{N}$ with $p\geq 2$. There exists a walk of length $pn-1$ from vertex $i$ to vertex $j$ in $D^m_n$ if and only if there exists a walk of length $n-1$ from $i$ to $j$ in $D^m_n$.
In such cases, the total number of walks of length $pn-1$ from $i$ to $j$ in $D^m_n$ is exactly $m^{p - 1}$.
\end{lem}

\proof By the previous lemma, a walk of length $pn-1$ from $i$ to $j$ exists if and only if there exists a walk of length $n - 1$ from $i$ to $j$. Each such walk of length $pn-1$ must include exactly $p-1$ cycles $C^{r_1}, \ldots, C^{r_{p - 1}}$, which can be independently selected from the $m$ available cycles in $D^m_n$. Since repetitions are allowed, there are $m^{p - 1}$ such combinations.\endproof

\textit{This result quantifies how the number of long walks in the digraph scales with the number of available cycles. It also illustrates the regularity of the structure: the repetition of cycles contributes multiplicatively to the count of admissible walks.}

We now present a set of results that characterize and count walks in $D^m_n$, which will play a key role in the proofs in the next section.

\begin{prop} \label{walksmtimes} For all $i, j \in V$, the number of walks in $D^m_n$ of length $2n - 1$ from vertex $i$ to vertex $j$ is exactly $m$ times the number of walks of length $n - 1$ from $i$ to $j$.
\end{prop}

\proof Let $i, j \in \{1, \ldots, m(n - 1) + 1\}$. By Lemmas \ref{walksn-1} and \ref{walksn-1a}, either there is no walk of length $n-1$ from $i$ to $j$ in $D^m_n$, or there exists a unique such walk. Furthermore, by Lemma \ref{walks2n-1}, there exists a walk of length $2n - 1$ from $i$ to $j$ if and only if there is a walk of length $n - 1$ between those vertices. In those cases, the number of walks of length $2n - 1$ is exactly $m$.\endproof

This result extends naturally to walks of longer lengths, as shown below.

\begin{prop} \label{walksm^p-1times} Let $p\in\mathbb{N}$ with $p\geq 2$. For all $i, j \in V$, the number of walks in $D^m_n$ of length $pn - 1$ from vertex $i$ to vertex $j$ is $m^{p - 1}$ times the number of walks of length $n-1$ from $i$ to $j$.\end{prop}

\proof Let $i,j$ in $\{1,\ldots,m(n-1)+1\}$. As before, by Lemmas \ref{walksn-1} and \ref{walksn-1a}, there is either no walk of length $n-1$ from $i$ to $j$, or exactly one. Lemma \ref{walkspn-1} states that such a walk of length $pn-1$ exists if and only if a walk of length $n-1$ exists between $i$ and $j$, and in those cases, there are exactly $m^{p-1}$ such walks of length $pn-1$.\endproof

\begin{lem} \label{walksn}
For every $i \in V \setminus \{1\}$, there exists exactly one walk in $D^m_n$ of length $n$ from $i$ to itself.\end{lem}

\proof Since $i\in V\setminus  \{1\}$, there exist $k\in\{1,\ldots,m\}$ and $\ell\in\{2,\ldots, n\}$ such that 
$i=(k-1)(n-1)+\ell$. The shortest walk from $i$ to itself is $\langle i\rangle$ and has length $0$. The walk
\begin{align*}P=&\Big\langle i=(k-1)(n-1)+\ell,(k-1)(n-1)+(\ell+1),\ldots, (k-1)(n-1)+n,1,\\
&(k-1)(n-1)+2,\ldots, (k-1)(n-1)+(\ell-1), (k-1)(n-1)+\ell =i\Big\rangle\end{align*}
has length $n$ and stays within the cycle $C^k$, returning to $i$. Any other walk from $i$ to itself must include one or more additional cycles, and can be written as 
$$Q_{1i}\circ C^{r_s}\circ\ldots\circ C^{r_1}\circ Q_{i1},$$ 
where $Q_{i1}$ and $Q_{1i}$ are the shortest paths in $C^k$ from $i$ to $1$ and from $1$ to $i$, respectively, $s\in\mathbb N$ and $r_1, \ldots, r_s \in \{1, \ldots, m\}$. These walks from $i$ to itself have length $(s+1)n$, strictly greater than $n$.
\endproof

\begin{lem} \label{walksn-2} For every $k\in\{1,\ldots,m\}$, there exists exactly one walk in $D^m_n$ of length $n-2$ from vertex $(k-1)(n-1)+2$ to $(k-1)(n-1)+n$.
\end{lem}

\proof Let $k\in\{1,\ldots,m\}$, $i=(k-1)(n-1)+2$ and $j=(k-1)(n-1)+n$. The shortest walk from $i$ to $j$ in $D^m_n$ is
$$P_{ij}=\Big\langle i=(k-1)(n-1)+2,(k-1)(n-1)+3,\ldots, (k-1)(n-1)+n =j\Big\rangle,$$
with length $n-2$. Any alternative walk from $i$ to $j$ must include detours through vertex $1$, and can be written as
$$P_{ij}\circ Q_{1i}\circ Q_{i1}$$
or
$$P_{ij}\circ Q_{1i}\circ C^{r_s}\circ\ldots\circ C^{r_1}\circ Q_{i1},$$ where $Q_{i1}$ and $Q_{1i}$ are the shortest paths in $C^k$ from $i$ to $1$ and from $1$ to $i$, respectively, $s\in\mathbb N$ and $r_1,\ldots, r_s\in\{1,\ldots,m\}$. These walks from $i$ to $j$ have lengths $n+|P_{ij}|$ and $(s+1)n+|P_{ij}|$, respectively, both strictly greater than $n-2$. \endproof

\begin{lem} \label{walks2n-2} For every $k\in\{1,\ldots,m\}$, there exists exactly one walk in $D^m_n$ of length $2n-2$ from vertex $(k-1)(n-1)+2$ to $(k-1)(n-1)+n$.
\end{lem}

\proof Let $k\in\{1,\ldots,m\}$, $i=(k-1)(n-1)+2$ and $j=(k-1)(n-1)+n$. As shown in Lemma \ref{walksn-2}, the shortest walk from $i$ to $j$ is
$$P_{ij}=\Big\langle i=(k-1)(n-1)+2,(k-1)(n-1)+3,\ldots, (k-1)(n-1)+n =j\Big\rangle,$$
with length $n-2$. If we append the path $Q_{1i} \circ Q_{i1}$, the shortest round trip from $i$ to $1$ and back, the total walk
$$
P_{ij} \circ Q_{1i} \circ Q_{i1}
$$
has length $n+|P_{ij}|=2n-2$. Any longer walk from $i$ to $j$ that includes one or more additional cycles must exceed this length. \endproof

\section{Adjacency matrices of oriented Dutch windmill graphs}  

\hspace{0.5cm} Let $M$ denote the adjacency matrix of the oriented Dutch windmill graph $D^m_n$. It is straightforward to verify that
 $\mathrm{rank}(M)=m(n-2)+2$. Moreover, $M$ is invertible if and only if $m=1$, in which case $M^{-1}=M^T$.

In the remainder of this section, we focus on the case $m > 1$. We will show that $M$ is not group invertible, determine that its Drazin index is $n - 1$, and derive an explicit expression for its Drazin inverse.

Recall that the $(i, j)$-entry of $M^k$ gives the number of walks of length $k$ from vertex $i$ to vertex $j$ in $D^m_n$.

The following theorem plays a central role in determining the Drazin index of $M$ and in deriving a closed-form expression for its Drazin inverse.

\begin{thm} \label{adj_pwr} Let $M$ be the adjacency matrix of the oriented Dutch windmill graph $D^m_n$. Then the following statements hold:
\begin{enumerate}
\item $M^{2n-1}=mM^{n-1}$;
\item $M^{n-1}\neq 0$;
\item $M^{2n-2}\neq mM^{n-2}$;
\item $M^n\neq mI$, where $I$ is the identity matrix of order $m(n-1)+1$;
\item $M^{n^2-1}=m^{n-1}M^{n-1}$.
\end{enumerate}
\end{thm}

\proof 
(1.) By Proposition \ref{walksmtimes}, for any pair of vertices $i, j \in V$, the number of walks of length $2n - 1$ from $i$ to $j$ in $D^m_n$ is exactly $m$ times the number of walks of length $n - 1$ from $i$ to $j$. Hence, we have
$$
(M^{2n-1})_{ij} = m (M^{n-1})_{ij}
\quad \text{for all } i,j,
$$
which gives $M^{2n - 1} = m M^{n - 1}$.

(2.) Lemma \ref{walksn-1} ensures that there exist walks of length $n - 1$ in $D^m_n$, so $M^{n - 1} \neq 0$.

(3.) Suppose, for contradiction, that $M^{2n - 2} = m M^{n - 2}$. Consider the entry in position $(2,n)$. By Lemma \ref{walksn-2}, there is exactly one walk of length $n - 2$ from vertex 2 to vertex $n$, so $(M^{n - 2})_{2,n} = 1$, and hence $(m M^{n - 2})_{2,n} = m$. However, Lemma \ref{walks2n-2} states that there is exactly one walk of length $2n - 2$ from 2 to $n$, so $(M^{2n - 2})_{2,n}=1$, contradicting the assumption that the entries match. Therefore, $M^{2n - 2} \neq m M^{n - 2}$.

(4.) From Lemma \ref{walksn}, we know that all diagonal entries of $M^n$, except for the entry in position $(1,1)$, are equal to $1$. Since $m > 1$, it follows that $M^n \neq mI$, where $I$ is the identity matrix of the same order as $M$.

(5.) Finally, by Proposition \ref{walksm^p-1times}, the number of walks of length $n^2-1$ from any vertex $i$ to any vertex $j$ is $m^{n - 1}$ times the number of walks of length $n-1$ between those vertices. Therefore,
$$
(M^{n^2-1})_{ij} = m^{n-1} (M^{n-1})_{ij} \quad \text{for all } i, j,
$$
which implies $M^{n^2-1}=m^{n-1} M^{n-1}$.\endproof

\textit{This result shows that the walk structure of the digraph imposes strong algebraic constraints on its adjacency matrix. The relation $M^{2n-1} = m M^{n-1}$, for instance, reflects the fact that all walks of length $2n-1$ arise from extending walks of length $n-1$ by a full cycle, with $m$ choices corresponding to the $m$ available cycles.}

\subsection{Drazin index}

\hspace{0.5cm} From Theorem \ref{adj_pwr}(1.), we know that the polynomial
$$
\varphi(\lambda) = \lambda^{2n-1} - m\lambda^{n-1} = \lambda^{n-1}(\lambda^n - m)
$$
annihilates the adjacency matrix $M$. Furthermore, Theorem \ref{adj_pwr}(2.)-(4.) shows that the polynomials
$$
\varphi_1(\lambda) = \lambda^{n-1}, \quad 
\varphi_2(\lambda) = \lambda^{2n-2} - m\lambda^{n-2} = \lambda^{n-2}(\lambda^n - m), \quad
\varphi_3(\lambda) = \lambda^n - m
$$
are not annihilating polynomials of $M$. Therefore, $\varphi(\lambda)$ is a minimal-degree annihilating polynomial of the form $\lambda^{n-1} g(\lambda)$, with $\gcd(\lambda, g(\lambda)) = 1$ and $\psi_A(\lambda)=\lambda^{n-1} h(\lambda)$ for some $h(\lambda$) with $\gcd(\lambda, h(\lambda)) = 1$. From these facts, we derive the following result.

\begin{thm} \label{ind} The Drazin index of the adjacency matrix of the oriented Dutch windmill graph $D^m_n$ is $n - 1$.
\end{thm}

\textit{The fact that the Drazin index equals $n-1$ reflects the minimal walk length required to fully capture the cyclic structure of the digraph through its adjacency powers. This highlights how the depth of the digraph's feedback structure determines the algebraic complexity of its pseudo-inverse.}

It follows from the result above that the adjacency matrices of oriented Dutch windmill graphs do not admit a group inverse, as their Drazin index satisfies $n-1\geq 2$.

\subsection{Drazin inverse}

\hspace{0.5cm} The Drazin inverse of a singular square matrix $A$ can be obtained via the so-called core-nilpotent decomposition, as follows. Given a singular square matrix $A$, there exist invertible matrices $U$ and $C$, and a nilpotent matrix $N$ such that $N^k = 0 \neq N^{k-1}$, for some positive integer $k$, and
$$
A = U \begin{bmatrix} C & 0 \\ 0 & N \end{bmatrix} U^{-1}.
$$
In this case, the Drazin inverse of $A$ is given by
$$
A^D = U \begin{bmatrix} C^{-1} & 0 \\ 0 & 0 \end{bmatrix} U^{-1},
$$
with $\mathrm{ind}(A) = k$.

Now, let $A$ be a square matrix and suppose that $k = \mathrm{ind}(A)$, and that there exists a matrix $X$ such that $A^k = A^{k+1}X$. Then, by applying the core-nilpotent decomposition, it follows that the Drazin inverse of $A$ can be expressed as
$$
A^D = A^k X^{k+1}.
$$
This expression applies in particular when an annihilating polynomial of $A$ has the form $\chi_A(\lambda) = \lambda^k g(\lambda)$, with $\gcd(\lambda, g(\lambda)) = 1$ and $g(\lambda) = g_0 + g_1 \lambda + \cdots + g_\ell \lambda^\ell$, where $\ell \geq 1$. From
$$
0 = A^k g(A) = A^k (g_0 I + g_1 A + \cdots + g_\ell A^\ell),
$$
we obtain
$$
g_0 A^k = A^{k+1}(-g_1 I - g_2 A - \cdots - g_\ell A^{\ell - 1}),
$$
which gives
$$
A^k = A^{k+1} X, \quad \text{with} \quad X = -\frac{1}{g_0}(g_1 I + g_2 A + \cdots + g_\ell A^{\ell - 1}).
$$

Applying this to our case, from Theorem \ref{adj_pwr}(1.) we know that
$$
\varphi(\lambda) = \lambda^{n - 1}(\lambda^n - m)
$$
is an annihilating polynomial of $M$, with $g(\lambda) = \lambda^n - m$, satisfying $\gcd(\lambda, g(\lambda)) = 1$. Thus, since $\mathrm{ind}(M) = n-1$ by Theorem \ref{ind}, we set
$$
X = -\frac{1}{-m} M^{n - 1} = \frac{1}{m} M^{n - 1},
$$
so that
$$
M^{n - 1} = M^n X.
$$
Then the Drazin inverse of $M$ is
$$
M^D = M^{n - 1} X^n = \frac{1}{m^n} M^{n - 1} M^{(n - 1)n} = \frac{1}{m^n} M^{n^2 - 1}.
$$
From Theorem \ref{adj_pwr}(5), it follows that
$$
M^D = \frac{1}{m^n} m^{n-1} M^{n-1} =\frac{1}{m} M^{n-1}.
$$
Hence, the Drazin inverse of $M$ is simply a scaled version of $M^{n - 1}$, where the scalar factor is $\frac{1}{m}$. To determine its nonzero entries, we use Lemmas \ref{walksn-1} and \ref{walksn-1a}, which describe the structure of walks of length $n-1$ in $D^m_n$. These results imply that the only nonzero entries of $M^D$, all equal to $\frac{1}{m}$, occur precisely at the positions $(i, j)$ satisfying one of the following:
\begin{enumerate}
\item[(i.)]  $i=1$ and $j=(k-1)(n-1)+n$, for some $k \in \{1,\ldots,m\}$;
\item[(ii.)] $i=(k-1)(n-1)+2$ and $j=1$, for some $k \in \{1, \ldots, m\}$;
\item[(iii.)] $i=(k-1)(n-1)+\ell$ and $j=i-1$, for some $k\in\{1,\ldots,m\}$ and $\ell\in\{3,\ldots, n\}$;
\item[(iv.)] $i=(k-1)(n-1)+\ell$ and $j=(r-1)(n-1)+(\ell-1)$, for distinct $k,r\in\{1,\ldots,m\}$ and $\ell\in\{3,\ldots, n\}$. 
\end{enumerate}
These positions correspond exactly to the vertex pairs connected by walks of length $n - 1$ in the oriented Dutch windmill graph $D^m_n$, as detailed in the aforementioned lemmas. 

From the preceding results, we obtain the following theorem.

\begin{thm} Let $M$ be the adjacency matrix of the oriented Dutch windmill graph $D^m_n$, with $m \geq 2$ and $n \geq 3$. Then the Drazin inverse of $M$ is given by
$$
M^D = \frac{1}{m} M^{n-1}.
$$
Moreover, the only nonzero entries of $M^D$ are equal to $\frac{1}{m}$, and occur precisely at the positions $(i, j)$ satisfying one of the following:
\begin{enumerate}
\item  $i=1$ and $j=(k-1)(n-1)+n$, for some $k \in \{1,\ldots,m\}$;
\item $i=(k-1)(n-1)+2$ and $j=1$, for some $k \in \{1, \ldots, m\}$;
\item $i=(k-1)(n-1)+\ell$ and $j=i-1$, for some $k\in\{1,\ldots,m\}$ and $\ell\in\{3,\ldots, n\}$;
\item $i=(k-1)(n-1)+\ell$ and $j=(r-1)(n-1)+(\ell-1)$, for distinct $k,r\in\{1,\ldots,m\}$ and $\ell\in\{3,\ldots, n\}$. 
\end{enumerate}
These positions correspond exactly to the pairs of vertices connected by walks of length $n - 1$ in $D^m_n$.
\end{thm}

This result shows that the Drazin inverse reflects precise walk relationships in the digraph. Each nonzero entry corresponds to a unique walk of length $n-1$ between vertices, highlighting the close interplay between combinatorial walk structure and algebraic pseudo-inverses.

\section*{Conclusion}

\hspace{0.5cm} We have determined the Drazin index and provided explicit expressions for the Drazin inverse of the adjacency matrices of oriented Dutch windmill graphs. The results reveal how the combinatorial structure of these digraphs governs their algebraic invertibility, and how the walk enumeration approach leads naturally to polynomial identities involving matrix powers.

The constructive framework developed here extends existing results for paths, cycles, and bipartite graphs, while illustrating the algebraic richness of recursively structured digraphs. It highlights the direct correspondence between the number and arrangement of directed cycles and the algebraic relations satisfied by their adjacency matrices.

Beyond its theoretical implications, this methodology can be applied to discrete dynamical models and symbolic computation of generalized inverses in structured networks. Future work may explore similar approaches for broader classes of coalescent digraphs or weighted versions of Dutch windmill structures, where the combination of algebraic and combinatorial perspectives could yield further insights into the spectral analysis of complex discrete systems.

\section*{Acknowledgements}
{This research was partially financed by Portuguese Funds
through FCT (Funda\c c\~ao para a Ci\^encia e a Tecnologia) within the Project UID/00013/2025.}
\section*{Data Availability}
No datasets were generated or analyzed during the current study. All results are theoretical and fully contained within the manuscript.
\section*{Conflict of interest}
The authors declare that they have no conflict of interest.

\bibliographystyle{plain}
\bibliography{BibliographyODW}

\bigskip
\end{document}